\documentclass{compositio}

\usepackage{amsmath}
\usepackage{amsfonts}
\usepackage{amsthm}
\usepackage{mathtools}
\usepackage{amssymb}
\usepackage[shortlabels]{enumitem}
\usepackage[style=alphabetic-verb, maxcitenames=50,maxnames=50, backend= biber,url = false]{biblatex}
\AtEveryBibitem{%
  \clearfield{note}%
}
\usepackage[bookmarks=true, colorlinks=true, linkcolor=blue, citecolor= red,hypertexnames=false,hyperindex,breaklinks]{hyperref}
\PassOptionsToPackage{linktocpage}{hyperref}

\newtheorem{thm}{Theorem}[section]
\newtheorem{lem}[thm]{Lemma}
\newtheorem{cor}[thm]{Corollary}

\newtheorem{prop}[thm]{Proposition}
\newtheorem{definition}[thm]{Definition}

\newtheorem{remark}[thm]{Remark}
\newtheorem*{thm*}{Theorem}
\newtheorem{fact}[thm]{Fact}

\renewcommand{\epsilon}{\varepsilon}

\DeclarePairedDelimiterX{\norm}[1]{\lVert}{\rVert}{#1}
\DeclarePairedDelimiterX{\abs}[1]{\lvert}{\rvert}{#1}

\DeclarePairedDelimiterX\braket[2]{\langle}{\rangle}{#1\,\delimsize\vert\,\mathopen{}#2}

\DeclarePairedDelimiterX{\gen}[1]{\langle}{\rangle}{#1}
\newcommand{\bbQ}{\mathbb{Q}}
\newcommand{\R}{\mathcal{R}}

\newcommand{\cstar}{$\mathrm{C}^*$}
\usepackage{graphicx}
\renewcommand{\cal}[1]{\mathcal{#1}}
\newcommand{\dminus}{\buildrel\textstyle\ .\over{\hbox{ \vrule height3pt depth0pt width0pt}{\smash-}}}
\newcommand{\spec}{\operatorname{spec}}

\newcommand{\MIPstar}{\mathsf{MIP}^*}
\newcommand{\MIPco}{\mathsf{MIP}^{\mathsf{co}}}
\newcommand{\RE}{\mathsf{RE}}
\newcommand{\coRE}{\mathsf{coRE}}

\title[The Universal Theory of Locally Universal $\mathrm{II}_1$ factors is not Computable]{The Universal Theory of Locally Universal Tracial von Neumann Algebras is not Computable}
\author{Jananan Arulseelan}
\email{jananan@iastate.edu}
\address{Department of Mathematics, Iowa State University, 396 Carver Hall, 411 Morrill Road, Ames, IA 50011, USA, \url{https://sites.google.com/view/jananan-arulseelan}}

\author{Aareyan Manzoor}
\email{a2manzoo@uwaterloo.ca}
\address{Department of Mathematics, University of Waterloo, University Avenue, Waterloo, Ontario, Canada, \url{https://aareyanmanzoor.github.io} }

\begin{document}

 \begin{abstract}
Building on Lin’s breakthrough $\MIPco = \coRE$ and an encoding of non-local games as universal sentences in the language of tracial von Neumann algebras, we show that locally universal tracial von Neumann algebras have undecidable universal theories. This implies that no such algebra admits a computable presentation.  Our results also provide, for the first time,  examples of separable $\mathrm{II}_1$ factors without computable presentations, and in fact yield a broad family of them, including McDuff factors, factors with neither property~$\Gamma$ nor property~(T), and property~(T) factors.   We also obtain analogous results for locally universal semifinite von Neumann algebras and tracial $C^*$-algebras.  The latter provides strong evidence for a negative solution to the Kirchberg embedding problem.  We discuss how these are obstructions to approximation properties in the class of tracial/semifinite von Neumann algebras.
\end{abstract}

\maketitle

\section{Introduction}

The Connes Embedding Problem (CEP), first posed as an offhand comment in \cite{connes_classification_1976}, asked whether the hyperfinite tracial von Neumann algebra $\R$ is \textbf{locally universal} for tracial von Neumann algebras.  That is, does every separable tracial von Neumann algebra embed into an ultrapower of $\R$?  At a conceptual level, CEP asked whether all tracial von Neumann algebras can be faithfully approximated by finite-dimensional matrix algebras.  Because of its connections to Kirchberg’s QWEP Conjecture, Tsirelson’s Problem in quantum information, the Microstates Conjecture in free probability, and the hyperlinearity problem for groups, CEP stood for decades as one of the central open problems in operator algebras.  

After more than 40 years, CEP was resolved negatively by Ji--Natarajan--Vidick--Wright--Yuen \cite{ji_mip_2020}.  Their breakthrough imported ideas from quantum complexity theory, showing that the complexity classes $\MIPstar$ and $\RE$ coincide. Building further, Goldbring and Hart \cite{goldbring_universal_2024} used model-theoretic techniques and part of the $\MIPstar = \RE$ machinery to show that the universal theory of $\R$ is undecidable, thereby strengthening the negative solution to CEP and producing non-Connes embeddable tracial von Neumann algebras in many natural classes.  Their work also sparked a significant volume of work in computable model theory of operator algebras.

 Lin recently announced a proof of the dual result $\MIPco = \coRE$ in a preprint \cite{lin_mipco_2025}.  From an operator algebraic perspective, this is arguably more striking: while $\MIPstar=\RE$ yields results about the class of Connes embeddable tracial von Neumann algebras, $\MIPco = \coRE$ yields results about the class of \textbf{all} tracial von Neumann algebras.  In \cite{manzoor_invariant_2025}, the second author suggested that $\MIPco=\coRE$ should imply that the class of tracial von Neumann algebras admits no uniform approximation property.  Our main result is a model theoretic formulation of this prediction:

\begin{thm*}
    Suppose $\MIPco=\coRE$, then for each Turing machine $M$, there is a restricted universal sentence $\sigma_M$ in the language of tracial von Neumann algebras, computable from the description of $M$, such that:
    \begin{itemize}
        \item If $M$ does not halt, then $\displaystyle \sup_{\cal N}\, \sigma_M^{\cal N}=1$.
        \item If $M$ halts, then $\displaystyle \sup_{\cal N}\, \sigma_M^{\cal N}\le \tfrac12$.
    \end{itemize}
    Here the supremum ranges over all separable tracial von Neumann algebras $\cal N$. 
\end{thm*}

This rules out any approximation property that would allow one to algorithmically approximate from below the value of universal sentences across all tracial von Neumann algebras (as Connes embeddability would have done).  Our proof proceeds by encoding the quantum commuting value of non-local games as universal sentences.  

It has long been known (see \cite[Example~6.4]{farah_model_2014}) that a locally universal tracial von Neumann algebra must exist.  One such algebra, Sherman’s $\mathcal S$ factor, is obtained by a compactness argument.  Being locally universal, $\mathcal S$ is non-Connes embeddable since CEP fails.  From our encoding we deduce:

\begin{thm*}
    The universal theory of any locally universal tracial von Neumann algebra is not computable.
\end{thm*}

\begin{thm*}
   No locally universal tracial von Neumann algebra admits a computable presentation.
\end{thm*}

The above theorem provides, to our knowledge, the first  examples of separable $\mathrm{II}_1$ factors (or even tracial von Neumann algebras) with \textbf{no} computable presentations. It is easily seen via cardinality considerations that the vast majority of separable $\mathrm{II}_1$ factors do not admit computable presentations.  There are known examples of $\mathrm{II}_1$ factors whose so-called ``standard presentations'' are not computable, but this does not rule out the possibility that they admit some other presentation which is computable. Our methods yield not just a single example, but an entire family: by varying the construction, we obtain  $\mathrm{II}_1$ factors without computable presentations across natural subclasses, including McDuff factors, factors with neither property~$\Gamma$ nor property~(T), and factors with property~(T).  

Using the recent framework for continuous model theory of general von Neumann algebras due to the first author, we also obtain analogous results for locally universal semifinite von Neumann algebras.  In particular, the class of semifinite von Neumann algebras cannot admit approximation properties that apply uniformly to the entire class.  

We obtain analogous results for tracial \cstar-algebras. These results provide strong evidence for a negative solution to the Kirchberg Embedding Problem.  Indeed, the Cuntz algebra $\mathcal O_2$ is known to admit a computable presentation \cite{fox_locally_2024}.  If our methods could be extended to show that locally universal \cstar-algebras never admit computable presentations, it would follow that $\mathcal O_2$ is not locally universal, resolving the Kirchberg Embedding Problem in the negative.  
\subsection*{Acknowledgements}
We would like to thank Isaac Goldbring for providing several helpful comments and corrections on an earlier draft and, in particular, for pointing us to tracial \cstar-algebras, ultimately leading to Section~\ref{sec:4.3}. 
We would also like to thank Junqiao Lin for helpful discussions about $\mathsf{MIP}^{\mathsf{co}}=\mathsf{coRE}$.  Finally, we thank the anonymous referees for helpful comments.
\section{Preliminaries}
We work in the framework of continuous logic \cite{yaacov_model_2008}, specialized to tracial
von Neumann algebras (see \cite{farah_model_2014-1}). See \cite{goldbring_model_2023} for a detailed discussion of continuous model theory in operator algebras.

The language of tracial von Neumann algebras has symbols for the $*$-algebra operations, rational
scalars, and a predicate interpreted as the $2$-norm $\|x\|_2=\tau(x^*x)^{1/2}$; the ambient
metric is is given by $d(x,y) = \|x-y\|_2$. 

Formulas are built from from expressions in the language (e.g., $\|p(\bar x)\|_2$ for $*$-polynomials $p$)
using continuous connectives (max, min, rational scalings, truncated subtraction, etc.) and quantifiers (inf and sup).  With respect to connectives, the standard practice is to take all continuous functions, but, as discussed below, it suffices to take a computable dense subset of these (see \cite[Section 2]{goldbring_universal_2024}).
A \textbf{sentence} is a formula with no free variables; it evaluates as a real number, with $0$
nominally corresponding to truth. The value of a sentence $\sigma$ in a structure $\mathcal{N}$ is denoted by $\sigma^{\mathcal{N}}$. A theory is a collection of sentence. For example, $T_{\mathrm{II}_1}$ is the theory of $\mathrm{II}_1$ factors (see \cite[Proposition 3.4]{farah_model_2014-1}): if all the sentences in this evaluate to $0$ in some structure $\mathcal{N}$, then $\mathcal{N}$ is a $\mathrm{II}_1$ factor.

A \textbf{universal sentence} is one of the form $\sup_{\bar x}\psi(\bar x)$ with $\psi$ quantifier-free. The universal theory of $\mathcal{N}$ is the map sending universal sentences $\sigma$ to $\sigma^\mathcal{N}$. An example of such a sentence is $\sup_x \bigl|\norm{x}_2-\norm{x^*}_2\bigr|$, which, of course, evaluates to $0$ on every tracial von Neumann algebra.

For computability-theoretic reasons, one often works with 
\textbf{restricted universal sentences}, namely universal sentences built from atomic formulas using 
only continuous connectives from a computable dense set which includes truncated subtraction and rational scalings; 
see \cite{goldbring_universal_2024}. Truncated subtraction is:
\[a\dminus b := \max(a-b,0),\]
which can be viewed as a way to test if $a\geq b$.

\begin{definition}
    Let $\cal N$ be a tracial von Neumann algebra. We say that the (universal) theory of $\cal N$ is \textbf{computable} if there is an algorithm which takes as inputs a restricted (universal) sentence $\sigma$ and a rational number $\delta > 0$ and returns rational numbers $a < b$ with $b - a < \delta$ and for which $\sigma^{\cal N} \in (a, b)$.
\end{definition}

We are also interested in the following weaker notion.

\begin{definition}
    The universal theory of a tracial von Neumann algebra $\cal N$ is said to be \textbf{weakly effectively enumerable} if one can effectively enumerate the sentences $\sigma \dminus r := \max\{\sigma - r, 0\}$, where $\sigma$ is a restricted universal sentence, $r \in \bbQ^{>0}$, and $\sigma^{\cal N} \leq r$.
\end{definition}

We also use the notion of computable presentation. Roughly speaking, a computable presentation of a separable tracial von Neumann algebra is a way of 
describing the algebra so that a computer can do all relevant computations. 

Let $\mathcal{N}$ be a separable tracial von Neumann algebra.

\begin{definition}
    \begin{itemize}
        \item A \textbf{presentation} is a pair $\cal N^\# = (\cal N, (a_n)_{n \in \mathbb{N}})$ where $\{ a_n \ : \ n \in \mathbb{N}\}$ is a subset of $\cal N$ that generates a $^*$-algebra which is a $\|\cdot\|_2$-dense subset of $\cal N$.
        \item Elements of the sequence $(a_n)_{n \in \mathbb{N}}$ are called \textbf{special points} of the presentation. 
        \item Elements of the form $p(a_{i_1}, \ldots, a_{i_k})$ where $p$ is a $^*$-polynomial with coefficients in $\mathbb{Q}(i)$ and $a_{i_1}, \ldots, a_{i_k}$ are special points are called \textbf{rational points} of the presentation.
        \item We say that a presentation $\cal N^\#$ is \textbf{computable} if there is an algorithm which, upon the input of a rational point $p$ of $\cal N^\#$ and a natural number $k$, returns a rational number $q$ such that $|\|p\|_2 - q| < 2^{-k}$.
    \end{itemize}    
\end{definition}
Intuitively, this means that $\cal N^\#$ generates $\cal N$ in such a way that it is \textbf{algorithmically possible} to compute the norm of a rational point from a description of how it is generated.  See \cite{goldbring_operator_2021} for more details on computable presentations of tracial von Neumann algebras.

\section{\texorpdfstring{MIP$^{co}$ = coRE}{MIPco = coRE} and its Consequences}

\subsection{Non-local Games}
In this section, we recall the connection between traces on group $C^*$-algebras and strategies for non-local games.  
For background and a more detailed discussion, see \cite[Sections~4--5]{goldbring_connes_2021}.  

\begin{definition}
A \textbf{synchronous non-local game} $\mathfrak{G}$ consists of:
\begin{itemize}
    \item a finite \textbf{question set} $Q$;
    \item a finite \textbf{answer set} $A$;
    \item a probability distribution $\mu$ on $Q \times Q$ (the questions sent to Alice and Bob);
    \item a \textbf{decider function} 
    \[
        D : A^2 \times Q^2 \to \{0,1\}, 
    \]
    where $D(a,b|x,y)$ specifies whether the verifier accepts the answers $a,b$ to questions $x,y$.
\end{itemize}
\end{definition}

A play of $\mathfrak{G}$ proceeds as follows: the verifier samples $(x,y)\sim \mu$,  
sends $x$ to Alice and $y$ to Bob, and receives replies $a,b\in A$.  
The players win if and only if $D(a,b|x,y)=1$. Alice and Bob can decide on a strategy beforehand but are not allowed to communicate during the actual game. 
Thus a \textbf{strategy} is encoded by conditional probabilities
\[
    \big(p(a,b|x,y)\big)_{a,b,x,y}.
\]

To analyze synchronous non-local games algebraically, consider the universal $C^*$-algebra $\mathcal A_{Q,A}$ generated by projections
\[
   \{\, e_x^a \ : \ x\in Q, \ a\in A \,\}
\]
subject to the relations
\[
    e_x^a e_x^b = \delta_{ab} e_x^a, 
    \qquad \sum_{a\in A} e_x^a = 1 \quad \text{for each } x\in Q.
\]
Families $\{e_x^a\}$ satisfying the above conditions are called projection-valued measures (PVMs). One may think of $e_x^a$ as the “quantum probability’’ of answering $a$ on question $x$.

Given a tracial state $\tau$ on $\mathcal A_{Q,A}$, we obtain a strategy
\[
    p(a,b|x,y)\ :=\ \tau(e_x^a e_y^b).
\]
Such a strategy is called a \textbf{synchronous quantum commuting strategy}.
This definition is equivalent to the standard definition by \cite[Proposition 5.6]{paulsen_estimating_2016}.

Let $\omega^s_{\mathrm{co}}(\mathfrak{G})$ denote the best winning probability of $\mathfrak{G}$ over all synchronous quantum commuting strategies.  
In \cite{lin_mipco_2025}, Lin announced a proof of the following:

\begin{fact}
    For each Turing machine $M$, there is a non-local game $\mathfrak{G}_M$ computable from the description of $M$ such that:
    \begin{itemize}
        \item If $M$ does not halt, then $\omega^s_{\mathrm{co}}(\mathfrak{G}_M)=1$;
        \item If $M$ halts, then $\omega^s_{\mathrm{co}}(\mathfrak{G}_M)\leq \tfrac12$.
    \end{itemize}
    In other words, $\mathsf{MIP}^\mathsf{co}=\mathsf{coRE}$.
\end{fact}
Note that above we are able to use the synchronous value by \cite[theorem 1.2]{lin_tracial_2023} or \cite[Corollary 3.2]{marrakchi_almost_2023}.

\subsection{Universal Sentences for Game Values}
We now explain how to convert the synchronous quantum commuting value of a synchronous non-local game into the maximum value of a universal sentence in the language of tracial von Neumann algebras. This section follows a lot of ideas from \cite[Section 3]{goldbring_universal_2024}.

Fix $n,m\ge 2$.  For each question $x\in\{1,\dots,n\}$ and each answer $a\in\{1,\dots,m\}$, let $e_x^a$ range over the unit ball of a tracial von Neumann algebra $(\cal N,\tau)$, with $\|y\|_2:=\tau(y^*y)^{1/2}$.  
Set
\[
\varphi_{n,m}(\bar e)\ :=\
\max\Big(
\max_{x,a}\|{e_x^a} - ({e_x^a})^*\|_2,\;
\max_{x,a}\|({e_x^a})^2 - e_x^a\|_2,\;
\max_{x}\Big\|\sum_{a=1}^m e_x^a - 1\Big\|_2
\Big),
\]
where $\bar e=(e_x^a)_{x,a}$.

Define
\[
X_{n,m}^{\cal N}:= \cal Z_{\cal N}(\varphi_{n,m}) = \{\bar e\in \cal N^{nm}:\varphi_{n,m}(\bar e)=0\},
\]
so that $X_{n,m}^{\cal N}$ is the set of $n$-tuples of $m$-outcome PVMs in $\cal N$.  As noted in the display, $X_{n,m}^{\cal N}$ is by definition the zeroset of $\varphi_{n,m}$ evaluated in $\cal N$.  For any choice of $n, m$, we see that $\varphi_{n,m}$ is a computable formula in the language of tracial von Neumann algebras.

In classical first-order logic, the zeroset of a formula is automatically definable.  
In continuous model theory, however, definability requires an additional stability condition: if $\varphi_{n,m}(\bar e)$ is small, then $\bar e$ must be close (in the metric) to $X_{n,m}^{\cal N}$.  
This is precisely the “almost-near’’ property of continuous logic definable sets.  We need this stability condition to be uniform in $\cal N$ as below.

\begin{lem}\label{AlmostNear}
    For any $m, n \geq 2$, the assignment $\cal N \mapsto X_{n,m}^{\cal N}$ is a definable set relative to the theory of tracial von Neumann algebras.  In other words, for every $\epsilon > 0$, there exists $\delta(\epsilon) > 0$ such that for all tracial von Neumann algebras $(\cal N, \tau)$ and $\overline{a} \in \cal N$, if $\varphi_{n,m}(\overline{a}) < \delta$ then $d(\overline{a}, X_{n,m}^{\cal N}) \leq \epsilon$.

    Furthermore, $\delta(\epsilon)$ is a computable function.
\end{lem}

\begin{proof}
This is the main result of \cite{de_la_salle_orthogonalization_2022}. Alternatively, this is essentially the stability result of \cite[Lemma~3.5]{kim_synchronous_2018}.  
While the statement in the latter is proved for $\cal N=M_d(\mathbb{C})$, the argument carries over verbatim to arbitrary tracial von Neumann algebras.  
\end{proof}

Given a synchronous non-local game $\mathfrak{G}=(Q,A,\mu,D)$ with $|Q|=n$ and $|A|=m$, and a tuple
$\bar e=(e_x^a)_{x\in Q,\ a\in A}\in X_{n,m}^{\cal N}$ inside a tracial von Neumann algebra $(\cal N,\tau)$, define
\[
  \psi_{\mathfrak{G}}(\bar e)\ :=\ \sum_{x,y\in Q}\mu(x,y)\ \sum_{a,b\in A} D(a,b|x,y)\ \tau\!\big(e_x^a\,e_y^b\big).
\]

\begin{lem}\label{lem:qc-as-def-sup}
For a synchronous non-local game $\mathfrak{G}$, the commuting-operator value satisfies
\[
  \omega^s_{\mathrm{co}}(\mathfrak{G})
  \;=\;
  \sup_{(\cal N,\tau)}\ \sup_{\bar e\in X_{n,m}^{\cal N}}\ \psi_{\mathfrak{G}}(\bar e),
\]
where the outer supremum ranges over all tracial von Neumann algebras $(\cal N,\tau)$ and
$X_{n,m}^{\cal N}$ denotes the set of $n$ families of $m$-outcome PVMs in $\cal N$.
\end{lem}

\noindent
Since $X_{n,m}^{\cal N}$ is definable and $\psi_{\mathfrak{G}}$ is a quantifier-free formula which is $1$-Lipshitz, the inner supremum 
$\sup_{\bar e\in X_{n,m}^{\cal N}}\psi_{\mathfrak{G}}(\bar e)$ is the value of a universal sentence evaluated in $\cal N$.  Since $\psi_{\mathfrak{G}}$ is $1$-Lipshitz, \cite[Proposition 2.2]{goldbring_universal_2024} implies that this universal sentence can be approximated by restricted universal formulae, and furthermore, this approximation by restricted formulae is independent of choice of $\cal N$.

\begin{proof}
($\geq$) Suppose $\bar e\in X_{n,m}^{\cal N}$ is a PVM in a tracial von Neumann algebra $(\cal N,\tau)$.  
By universality, there is a $*$-homomorphism $\pi:\mathcal A_{Q,A}\to \cal N$ sending the generating PVM in $\mathcal A_{Q,A}$ to $\bar e$.  
Pulling back the trace, $\tau\circ\pi$ is a tracial state on $\mathcal A_{Q,A}$.  
The corresponding correlation $p(a,b|x,y)=\tau(e_x^a e_y^b)$
is therefore a synchronous quantum commuting strategy, and its value is precisely $\psi_{\mathfrak{G}}(\bar e)$.  
This shows that every $\bar e$ contributes to $\omega^s_{\mathrm{co}}(\mathfrak{G})$, hence
\[
\omega^s_{\mathrm{co}}(\mathfrak{G}) \;\geq\; \sup_{(\cal N,\tau)}\ \sup_{\bar e\in X_{n,m}^{\cal N}}\ \psi_{\mathfrak{G}}(\bar e).
\]

($\leq$) Conversely, let $p$ be any synchronous quantum commuting strategy for $\mathfrak{G}$.  
By definition, there exists a tracial state $\tau$ on $\mathcal A_{Q,A}$ such that 
$p(a,b|x,y)=\tau(e_x^a e_y^b)$.  
Since the trace simplex of $\mathcal A_{Q,A}$ is compact convex, the value $\omega^s_{\mathrm{co}}(\mathfrak{G})$ is attained at some trace. The GNS of this is a tracial von Neumann algebra and the generating PVM of $\mathcal{A}_{Q,A}$ gets sent to a PVM in this algebra.
Thus
\[
\omega^s_{\mathrm{co}}(\mathfrak{G}) \;\leq\; \sup_{(\cal N,\tau)}\ \sup_{\bar e\in X_{n,m}^{\cal N}}\ \psi_{\mathfrak{G}}(\bar e)
\]
as desired.
\end{proof}
We can now prove the main theorem:

\begin{thm}\label{MIPcoEncoding}
For each Turing machine $M$, there is a restricted universal sentence $\sigma_M$ in the language of tracial von Neumann algebras, computable from the description of $M$, such that:
\begin{itemize}
    \item If $M$ does not halt, then $\displaystyle \sup_{\cal N}\, \sigma_M^{\cal N}=1$.
    \item If $M$ halts, then $\displaystyle \sup_{\cal N}\, \sigma_M^{\cal N}\le \tfrac12$.
\end{itemize}
Here the supremum ranges over all tracial von Neumann algebras $\cal N$. 
In other words, there is a computable reduction from the non-halting problem to the evaluation of (restricted) universal sentences.
\end{thm}
\begin{proof}
Fix a Turing machine $M$. By \cite{lin_mipco_2025}, there is a synchronous non-local game $\mathfrak{G}_M$ computable from $M$ such that 
$\omega^s_{\mathrm{co}}(\mathfrak{G}_M)=1$ if $M$ does not halt, and 
$\omega^s_{\mathrm{co}}(\mathfrak{G}_M)\le \tfrac12$ if $M$ halts.

By Lemma~\ref{lem:qc-as-def-sup}, there is a universal sentence $\varphi_M$ in the language of tracial von Neumann algebras with
\[
\sup_{\cal N} \varphi_M^{\cal N} \;=\; \omega^s_{\mathrm{co}}(\mathfrak{G}_M).
\]
In particular, if $M$ does not halt then $\sup_{\cal N} \varphi_M^{\cal N}=1$, while if $M$ halts then $\sup_{\cal N} \varphi_M^{\cal N}\leq \tfrac12$.

Finally, by Goldbring–Hart \cite[Prop.~2.2]{goldbring_universal_2024}, every universal sentence can be approximated by restricted universal sentences uniformly in $\mathcal{N}$, proving the result.
\end{proof}
\section{Applications to Computable Continuous Model Theory}

In this section, we present several interesting model-theoretic consequences of Theorem \ref{MIPcoEncoding}.  In \cite{goldbring_universal_2024}, Goldbring--Hart used a variant of Theorem \ref{MIPcoEncoding} that arose from $\mathsf{MIP}^*=\mathsf{RE}$ to prove parallel results for $\R$. 

\subsection{Locally Universal Models}

We begin by recalling some essential background from model theory.

\begin{definition}
    A tracial von Neumann algebra $\mathcal{S}$ is called \textbf{locally universal} if every separable tracial von Neumann algebra embeds in an ultrapower of $\mathcal{S}$.
\end{definition}

The main theorem of \cite{ji_mip_2020} shows that the hyperfinite II$_1$ factor $\R$ is not locally universal, thereby settling CEP in the negative.  On the other hand, Farah--Hart--Sherman build a separable locally universal tracial von Neumann algebra $\cal S$ in \cite{farah_model_2014}, thereby giving a ``poor man's resolution of the Connes Embedding Problem''.  Note that since CEP is resolved in the negative, $\cal S$ does not embed in any ultrapower of $\R$.  

We briefly sketch the construction of $\cal S$.  Taking tensor products or free products over representatives of all isomorphism classes of separable $\mathrm{II}_1$ factors, one obtains a very large (non-separable) $\mathrm{II}_1$ factor $\cal S_{0}$ containing (copies of) every $\mathrm{II}_1$ factor.  In particular, for every universal sentence $\sigma$ one has
$\sigma^{\cal S_{0}} = \sup_{\cal N} \sigma^{\cal N},$
where the supremum ranges over all separable $\mathrm{II}_1$ factors $\cal N$.  By the downward Löwenheim--Skolem theorem, there exists a separable elementary submodel $\mathcal{S}\subset \mathcal{S}_0$, meaning that $\mathcal{S}$ and $\mathcal{S}_0$ agree on the truth values of all sentences (in particular, of all universal sentences).  This means that, for each separable $\mathrm{II}_1$ factor $\mathcal{N}$ and universal sentence $\sigma$, we have $\sigma^{\mathcal{N}}\leq \sigma^{\mathcal{S}}$. This implies $\mathcal{N}$ embeds into an ultrapower of $\mathcal{S}$. Finally since every tracial von Neumann algebra embeds in a $\mathrm{II}_1$ factor (for example, by considering its free product with $\cal R$ \cite[theorem 3.4]{ueda_factoriality_2011}), we get that $\mathcal{S}$ is locally universal.

We denote by $\cal S$ the factor described by the construction above. Note that it has the same universal theory as any other locally universal tracial von Neumann algebra, so facts about its universal theory apply to every locally universal tracial von Neumann algebra.

\subsection{Results on Computability of Universal Theories}
It is already known that $\cal S$ has a weakly effectively enumerable universal theory (combining  \cite[Theorem 5.2 and 5.6]{goldbring_universal_2024}) In contrast, we now have our main theorem of this section.

\begin{thm}\label{SUndecidable}
    The universal theory any locally universal tracial von Neumann algebra is not computable.
\end{thm}

\begin{proof}
    By local universality of $\cal S$, we have
    \[
    \sup_{\cal N} \sigma^{\cal N} = \sigma^{\cal S}
    \]
    for every universal sentence $\sigma$ in the language of tracial von Neumann algebras.  Assume, by way of contradiction, that the universal theory of $\cal S$ is computable.  Then for any Turing machine $M$, we can computably find an interval $(a,b)$, with $a,b \in \bbQ$, $b-a < \frac{1}{8}$ such that $\sigma_{M}^{\cal S} \in (a,b)$, where $\sigma_{M}$ is the restricted universal sentence given by Theorem \ref{MIPcoEncoding}.  Since $\sigma_{M}^{\cal S}$ is precisely $\sup_{\cal N} \sigma_{M}^{\cal N}$, by Theorem \ref{MIPcoEncoding}, then we know $M$ halts if $1 \not\in (a,b)$ and $M$ does not halt otherwise.  This contradicts the undecidability of the Halting Problem.
\end{proof}
Note that by the Pavelka-style completeness theorem for continuous logic \cite{ben_yaacov_proof_2010} we have, for every sentence $\sigma$,
\[
\sup_{\mathcal{N}} \sigma^\mathcal{N} \;=\; \inf\{r \in \mathbb{Q} : T_{II_1} \vdash r \dminus \sigma\}.
\]
In words, the maximum value attained by $\sigma$ across all tracial von Neumann algebras coincides with the smallest rational $r$ such that the theory of tracial von Neumann algebras proves the inequality $r \geq \sigma$.  The left-hand side is precisely $\sigma^{\mathcal{S}}$ by local universality.  Since we have shown that computing this quantity is coRE-complete, it follows that there is no effective approximation scheme for such suprema.  In particular, neither side of the above identity can be algorithmically approximated from below for universal sentences. This should be interpreted as the class of tracial von Neumann algebras not having any nice approximation property encompass everything.

For instance, if Connes’ Embedding Problem had held, then the left-hand side could be approximated by optimizing over finite-dimensional matrix algebras, contradicting our undecidability result.  This non-approximability phenomenon should be compared with \cite[Proposition~3.6]{manzoor_invariant_2025}, where the second author established a related obstruction for the class of tracial von Neumann algebras.

Often, uncomputable universal theory results like the theorem above come with ultraproduct nonembedding results as in \cite{goldbring_universal_2024} and \cite{arulseelan_undecidability_2024}. However, by definition of $\cal S$, all separable factors embed in its ultrapower, so we do not have such results. But, in exchange, we get an arguably more interesting result:

\begin{thm}\label{SNoCompPres}
    No locally universal tracial von Neumann algebra admits a computable presentation.
\end{thm}
\begin{proof}
    By preceding discussions, $\cal S$ has a weakly effective enumerable universal theory. Now suppose it had a computable presentation, then by \cite[Proposition 2.7]{goldbring_universal_2024} we have that $\cal S$ has a computable universal theory. This contradicts Theorem \ref{SUndecidable}.
\end{proof}

\begin{remark}
    The above theorem provides, to our knowledge, the first  examples of separable $\mathrm{II}_1$ factors (or even tracial von Neumann algebras) with \textbf{no} computable presentations. It is easily seen via cardinality considerations that the vast majority of separable $\mathrm{II}_1$ factors do not admit computable presentations.  More precisely, since there are only countably many computable presentations and far more than countably many separable $\mathrm{II}_1$ factors, most separable $\mathrm{II}_1$ factors will not admit computable presentations.  There are known examples of $\mathrm{II}_1$ factors whose so-called ``standard presentations'' are not computable, but this does not rule out the possibility that they admit some other presentation which is computable.
\end{remark}

Recall from \cite{farah_model_2014} that given any locally universal factor $\cal S$, both $\cal S \otimes \R$ and $\cal S * \R$ are also locally universal.  The former is McDuff and the latter has neither property $\Gamma$ nor property (T). We can also embed $\cal S$ into a property (T) factor by \cite{chifan_embedding_2023}. Thus we have also demonstrated the first  examples of McDuff $\mathrm{II}_1$ factors, $\mathrm{II}_1$ factors with neither property $\Gamma$ nor property~(T), and $\mathrm{II}_1$ factors with property (T) that do not admit computable presentation.

% This may provide a post-hoc explanation for why operator algebraists have, to this point, been unable to find a satisfactory representation of a non-Connes embeddable $\mathrm{II}_1$ factor.  That is, the natural examples of non-Connes embeddable factors that we get from $\MIP^*=\RE$ (the locally universal factors) have no computable presentations. We note that some natural factors (like the group von Neumann algebra of the Higman group \cite{Thom2012MetricApproxHigman}) are expected not to be Connes embeddable, but $\MIP^*=\RE$ and related result will probably not be able to prove this one way or the other.

Our results may explain in retrospect why the poor-man's solution to the Connes Embedding Problem in \cite{farah_model_2014-1} has yet to yield direct operator-algebraic constructions of non-Connes-embeddable factors.  Our results say that the examples constructed via compactness cannot easily be used in computations.  We remark on the other hand that there are many factors that are conjecturally not Connes-embeddable but nevertheless have computable presentations.  For example, the group von Neumann algebra of the Higman group is expected to not be Connes-embeddable (cf. \cite{Thom2012MetricApproxHigman}) but admits a computable presentation by the decidability of the word problem for the Higman group \cite{diekert_laun_ushakov_higman_inP_2012} and \cite{goldbring_operator_2021}.  Even in this case, our results rule out a potential proof strategy for showing it is not Connes-embeddable: the group von Neumann algebra of the Higman group cannot be locally universal.

We now consider the semifinite von Neumann algebras using the framework of \cite{arulseelan_model_2025}.  

\begin{prop}
    If $\cal S$ is locally universal for tracial von Neumann algebras, then $\cal S \otimes B(\cal H)$ is locally universal for semifinite von Neumann algebras.
\end{prop}
\begin{proof}
    Let $\cal N$ be a semifinite von Neumann algebra.  Then $\cal N \cong \cal N_0 \otimes B(\cal H)$ for some $\mathrm{II}_1$ factor $\cal N_0$.  Then, by assumption, there is an embedding $\cal N_0 \hookrightarrow \cal S^{\cal U}$.  We also have $S^{\cal U} \hookrightarrow S^{\cal U} \otimes B(\cal H)$ via the map $(s_i)^\bullet \mapsto (s_i)^\bullet \otimes 1$.  We also have the diagonal embedding $B(\cal H) \hookrightarrow B(\cal H)^{\cal U}$ of $B(\cal H)$ into its generalized Ocneanu ultrapower.  Putting these together yields an embedding of $\cal N_0 \otimes B(\cal H) \cong \cal N$ into $(\cal S \otimes B(\cal H))^{\cal U}$, proving the claim.
\end{proof}

Thus, by \cite[Theorem 8.2]{arulseelan_model_2025} and \cite[Theorem 8.5]{arulseelan_model_2025} combined with Theorem \ref{SUndecidable} we have the following consequence.

\begin{cor}\label{IIinftyNotComputable}
    The universal theory of $\cal S \otimes B(\cal H)$ is not computable.
\end{cor}
In particular, by the same argument as in the tracial von Neumann algebras case, we get that the class of semifinite von Neumann algebras cannot have nice approximation properties encompass the entire class. We also have:

\begin{cor}\label{IIinftyNoCompPres}
    Locally universal semifinite von Neumann algebras do not admit computable presentations.
\end{cor}

It would be interesting to connect the model theory of type III von Neumann algebras (see \cite{arulseelan_totally_2025} and \cite{arulseelan_model_2025}) using the results of \cite{marrakchi_almost_2023} and \cite{lin_tracial_2023}, which both heavily use the Tomita-Takesaki theory.  Since the aforementioned model theories are built to accommodate the Tomita-Takesaki theory and the latter two papers are seen in this paper to be closely linked to the model theory of tracial von Neumann algebras, we suspect there to be rich connections here.  We leave this to future work.
\subsection{Consequences for \texorpdfstring{\cstar}{C*}-algebras}\label{sec:4.3}

A natural question is to what extent our methods extend to the setting of \cstar-algebras.  Since our proofs rely on traces, we work in the language of \textbf{tracial \cstar-algebras}.  We need a \cstar-algebra analogue of Lemma \ref{AlmostNear} which, in turn, depends on the results of either \cite{kim_synchronous_2018} or \cite{de_la_salle_orthogonalization_2022} who each use spectral projections to establish 2-norm stability of POVMs.

Fortunately, this problem essentially reduces to a simple unitary stability result for \cstar-algebras. Here, the norm control of polar decompositions is the correct \cstar-analog of spectral projection arguments in 2-norm.

%We can carry over the previous section’s conclusions using the following uniform stability estimate.

\begin{lem}\label{U-Stab}
Let $\mathcal A$ be a \cstar-algebra and $m \in \mathbb{N}$ with $m \geq 1$ be given and assume $v \in \mathcal A$ satisfies
\[
\max\bigl(\,\|v^*v-1\|,\ \|vv^*-1\|,\ \|v^m-1\|\,\bigr)\ \le\ \epsilon,
\]
where $\epsilon<2^{-m}$.
Then there exists a unitary $u \in \mathcal A$ such that $u^n=1$ and
\[
\|u-v\|\ \le\ 2^{\,m+2}\,\epsilon.
\]
\end{lem}

\begin{proof}
Define $u_0 := v(vv^*)^{-1/2}$. Then $u_0$ is a unitary satisfying $\norm{v-u_0}\leq \norm{(vv^*)^{1/2}-1}$.  Note by continuous functional calculus that $\norm{vv^*-1}\leq \epsilon$ implies $\spec(vv^*)\subset [1-\epsilon,1+\epsilon]$.  Again by functional calculus, we see that $\norm{(vv^*)^{1/2}-1}$ is bounded above by $\sup_{x\in [1-\epsilon,1+\epsilon]} |\sqrt{x}-1|$, which is, in turn, bounded above by $\epsilon$. Hence $\norm{v-u_0}\leq \epsilon$.

Next note that
\[\begin{aligned}
    \norm{u_0^m -1}&\leq \norm{u_0^m-v^m}+\norm{v^m-1}\\
                &\leq \Bigl(\sum_k\norm{u_0}^k \norm{v}^{m-k}\Bigr) \norm{u_0-v}+\epsilon\\
                &\leq 2^{m+1}\epsilon.
\end{aligned}\]
In the last line we used that $\norm{u}\leq 2$ and $\norm{v}\leq 1$.

Since $\norm{u_0^m-1}<2$, the spectrum of $u_0^m$ is contained in a union of disjoint arcs, each containing exactly one $m$th root of unity. Letting $f$ be the continuous map that collapses each of these arcs to the corresponding root of unity, we have $f(u_0)^m=1$.  A trigonometric argument shows that $\norm{u_0-f(u_0)}\leq 2^{m+2}\epsilon/m$. Now $u := f(u_0)$ is a unitary such that $\norm{u-v}\leq 2^{m+2}\epsilon$, as desired.
\end{proof}

Let $U_{n,m}^{\mathcal A}$ denote the set of $n$-tuples of unitaries of order $m$ in $\mathcal A$. The preceding lemma implies the mapping $\mathcal A \mapsto U_{n,m}^{\mathcal A}$ forms a definable set relative to the theory of \cstar-algebras. Write $X_{n,m}^{\mathcal A}$ for the set of $n$-tuples of PVMs with $m$ outcomes in $\mathcal A$. Fixing a primitive $m$th root of unity $\omega$, there is a bijection
\[
X_{n,m}^{\mathcal A}\ \longleftrightarrow\ U_{n,m}^{\mathcal A},\qquad
\text{given by }\
u_x=\sum_{a=0}^{n-1}\omega^{a}e_x^{a},\quad
e_x^{a}=\frac{1}{n}\sum_{c=0}^{n-1}\omega^{-ac}(u_x)^{c}.
\]
Since the bijection is given by the image of a term in the language of \cstar-algebras, it follows that the mapping $\mathcal A \mapsto X_{n,m}^{\mathcal{A}}$ forms a definable set relative to the theory of \cstar-algebras.  This definability of PVMs relative to the theory of \cstar-algebras may be of independent interest in the model theory of \cstar-algebras.

By the same argument as the tracial von Neumann algebra case, we have that
\[
\omega^s_{\mathrm{co}}(\mathfrak G)\ =\ \sup_{(\mathcal A,\tau)}\ \sup_{\overline{e}\in X_{n,m}^{\mathcal A}}\ \psi_{\mathfrak G}\bigl(\overline{e}\bigr).
\]
where the outer supremum ranges over tracial \cstar-algebras. Using $\mathsf{MIP}^\mathsf{co}=\mathsf{coRE}$ and arguing exactly as in the tracial von Neumann algebra case, we obtain:

\begin{thm}
The universal theory of locally universal tracial \cstar-algebras is not computable.
\end{thm}

\begin{thm}
Locally universal tracial \cstar-algebras do not admit computable presentations.
\end{thm}

We are thus tantalizingly close to a negative solution of the Kirchberg embedding problem. However, achieving such a full resolution using our methods would require removing the dependence on traces. An entirely new approach appears to be necessary.

\printbibliography

\end{document}